\documentclass[a4paper,12pt]{article}

\usepackage[ascii]{inputenc}
\usepackage{amsmath,amsfonts,amssymb,amsthm}
\usepackage{mathrsfs}
\usepackage{comment, color}
\newtheoremstyle{theorem}{5pt}{5pt}{\itshape}{}{\bfseries}{.}{.5em}{}

\theoremstyle{theorem}
\newtheorem{theorem}{Theorem}
\newtheorem{lemma}[theorem]{Lemma}
\newtheorem{corollary}[theorem]{Corollary}
\newtheorem*{remark*}{Remark}

\theoremstyle{definition}
\newtheorem{definition}[theorem]{Definition}

\usepackage{titlesec}
\titleformat{\section}
{\Large\bfseries}{\thesection}{1em}{\large}
\titlespacing*{\section}{0pt}{3.5ex plus 1ex minus .2ex}{2.3ex plus .2ex}

\titleformat{\subsection}
{\normalfont\bfseries}{\thesubsection}{1em}{\normalsize}
\titlespacing*{\section}{0pt}{3.5ex plus 1ex minus .2ex}{2.3ex plus .2ex}

\allowdisplaybreaks[3]

\begin{document}

\title{On Fourth and Higher Moments of Short \\ Exponential Sums Related to Cusp Forms}
\author{Anne-Maria Ernvall-Hyt\"onen and Esa V\!. Vesalainen\\
\normalsize Mathematics and Statistics\\[-1mm]
\normalsize {\AA}bo Akademi University}
\date{}
\maketitle

\begin{abstract}
We obtain upper bounds for the fourth and higher moments of short exponential sums involving Fourier coefficients of holomorphic cusp forms twisted by rational additive twists with small denominators.
\end{abstract}

\section{Introduction}

Let us consider a fixed holomorphic cusp form $F$ for $\mathrm{SL}(2,\mathbb Z)$ of weight $\kappa\in\mathbb Z_+$. Then it will have the usual Fourier expansion which we will normalize so that
\[F(z)=\sum_{n=1}^\infty a(n)\,n^{(\kappa-1)/2}\,e(nz)\]
for all $z\in\mathbb C$ with $\Im z>0$. With this normalization Deligne's estimate \cite{Deligne1974} says that $a(n)\ll d(n)\ll n^\varepsilon$, for positive integers $n$, and the Rankin--Selberg estimate \cite{Rankin1939, Selberg1940} says that, for $M\in\left[1,\infty\right[$,
\[\sum_{n\leq M}\left|a(n)\right|^2=A\,M+O(M^{3/5}),\]
where $A$ is a positive real constant only depending on $F$.

It is of great interest to study exponential sums weighted by Fourier coefficients.
Wilton \cite{Wilton1929} proved essentially square root cancellation for long linear sums and Jutila \cite{Jutila1987b} removed the logarithm in Wilton's estimate leading to the best possible upper bound
\[\sum_{n\leq M}a(n)\,e(n\alpha)\ll M^{1/2},\]
uniformly true for $M\in\left[1,\infty\right[$ and $\alpha\in\mathbb R$.
The case where $\alpha$ is a reduced fraction $h/k$ with a small denominator $k$ is very interesting, and provides an interesting analogue to the classical problems of studying the error terms in the Dirichlet divisor problem or the circle problem, see e.g.\ \cite{Jutila1987a}. Jutila \cite{Jutila1987a} proved the pointwise upper bound $\ll k^{2/3}\,M^{1/3+\varepsilon}$. When $M^{1/10}\ll k\ll M^{5/18}$, this has been improved to $k^{1/4}\,M^{3/8+\varepsilon}$ in \cite{Jaasaari--Vesalainen2015, Vesalainen2014}, based on short sum estimates from \cite{Ernvall-Hytonen--Karppinen2008}.

Jutila \cite{Jutila1987a} also obtained a mean square result analogous to a twisted mean square result for the divisor function in \cite{Jutila1985}, which in turn was in the spirit of earlier work of Cram\'er \cite{Cramer1922} for the divisor problem without twists. Crudely speaking, when $k\ll M^{1/2-\varepsilon}$, the size of the sum is proportional to $k^{1/2}\,M^{1/4}$ on average. More precisely,
\[\int\limits_M^{2M}\left|\sum_{n\leqslant x}a(n)\,e\!\left(n\,\frac hk\right)\right|^2\mathrm dx=C_F\,k\,M^{3/2}+O(k^2\,M^{1+\varepsilon})+O(k^{3/2}\,M^{5/4+\varepsilon}),\]
where $C_F$ is a positive real constant depending on $F$ only. When $k\ll M^{1/6-\varepsilon}$, it was proved in \cite{Vesalainen2014} following \cite{Tsang1992} and especially \cite{Ivic--Sargos2007} that the sum is of the same average order of magnitude in the sense of fourth moments also:
\[\int\limits_M^{2M}\left|\sum_{n\leqslant x}a(n)\,e\!\left(n\,\frac hk\right)\right|^4\mathrm dx
=C_F'\,k^2\,M^2+O(k^{11/4}\,M^{15/8+\varepsilon})+O(k^{13/6}\,M^{23/12+\varepsilon}),\]
where again $C_F'$ is a positive real constant only depending on the underlying cusp form.

We are interested here in the properties of the short linear sums
\[\sum_{M\leq n\leq M+\Delta}a(n)\,e(n\alpha),\]
where $M\in\left[1,\infty\right[$, $\Delta\in\left[1,M\right]$ and $\alpha\in\mathbb R$.
The best known upper bounds for such sums are due to the first author and Karppinen \cite{Ernvall-Hytonen--Karppinen2008} with a minor improvement by J\"a\"asaari and the second author \cite{Jaasaari--Vesalainen2015}. We will specifically study the case of a rational $\alpha$ with a small denominator. The study of these short exponential sums is a natural analogue of short interval considerations of error terms in classical analytic number theory. We note that estimates for short sums can also sometimes be used to reduce smoothing error in other arguments, as is done in \cite{Jutila1987b, Ernvall-Hytonen2008, Ernvall-Hytonen--Karppinen2008, Vesalainen2014, Jaasaari--Vesalainen2015}.

Jutila \cite{Jutila1984} considered the mean square of the error term in the Dirichlet divisor problem in short intervals. The method also applies to short sums of Fourier coefficients with $\Delta\ll M^{1/2}$, leading to square root cancellation on average, see Ivi\'c \cite{Ivic2009} and Wu and Zhai \cite{Wu--Zhai2013}. The second moment of short sums of Fourier coefficients with rational additive twists was studied in \cite{Ernvall-Hytonen2011, Vesalainen2014}. The square root cancellation still holds on average as long as $k\ll\Delta^{1/2-\varepsilon}$. 
More precisely, when $1\leq\Delta\ll M^{1/2}$ and $k\ll\Delta^{1/2-\varepsilon}$, we have
\[\int\limits_M^{2M}\,\left|\sum_{x\leq n\leq x+\Delta}a(n)\,e\!\left(n\,\frac hk\right)\right|^2\mathrm dx\ll M\,\Delta.\]
The second moment of longer short sums was estimated in \cite{Ernvall-Hytonen2015}. The moment estimates give rise to the conjecture that
\[\sum_{M\leqslant n\leqslant M+\Delta}a(n)\,e\!\left(n\,\frac hk\right)\ll\min\left(\Delta^{1/2}\,M^\varepsilon,k^{1/2}\,M^{1/4+\varepsilon}\right).\]

Ivi\'c \cite{Ivic2009} considered the fourth moment of the error term in the Dirichlet divisor problem in short intervals and obtained the expected upper bound when the interval was not too short. Wu and Zhai \cite{Wu--Zhai2013} observed that the same technique works for sums of Fourier coefficients. Our first goal here is to consider the fourth moment of short exponential sums with rational additive twists with small denominators in the spirit of \cite{Ivic2009}. Tanigawa and Zhai \cite{Tanigawa--Zhai2009} extracted a main term in the case of divisor function, but we do not attempt this.

Our second goal is to estimate general higher moments through large value estimates. This follows the consideration of large values of the error term in the Dirichlet divisor problem in short intervals in \cite{Ivic--Zhai2014} and the study of higher moments of rationally additive twisted moments of long sums of holomorphic cusp form coefficients in \cite{Vesalainen2014}, which in turn followed similar study for the moments of the error term in the Dirichlet divisor problem in~\cite{Ivic1983}.

\section{Notation}

We use standard asymptotic notation. If $f$ and $g$ are complex-valued functions defined on some set, say $\Omega$, then we write $f\ll g$ to signify that $\left|f(x)\right|\leq C\left|g(x)\right|$ for all $x\in\Omega$ for some implicit constant $C\in\mathbb R_+$. The notation $O(g)$ denotes a quantity that is $\ll g$, and $f\asymp g$ means that both $f\ll g$ and $g\ll f$. The letter $\varepsilon$ denotes a positive real number, whose value can be fixed to be arbitrarily small, and whose value can be different in different instances in a proof. All implicit constants are allowed to depend on $\varepsilon$, on the implicit constants appearing in the assumptions of theorem statements, and on anything that has been fixed. When necessary, we will use subscripts $\ll_{\alpha,\beta,\dots}$, $O_{\alpha,\beta,\dots}$, etc.\ to indicate when implicit constants are allowed to depend on objects $\alpha$, $\beta$, \dots

The numbers $a(1)$, $a(2)$, \dots\ will always denote the Fourier coefficients of a fixed holomorphic cusp form $F$ of even weight $\kappa\in\mathbb Z_+$ for the full modular group $\mathrm{SL}(2,\mathbb Z)$. The Fourier coefficients are normalized so that the Fourier expansion of the cusp form is
\[F(z)=\sum_{n=1}^\infty a(n)\,n^{(\kappa-1)/2}\,e(nz)\]
for $z\in\mathbb C$ with $\Im z>0$. All implicit constants are allowed to depend on $F$.

The function $w(x)$ is a particular kind of smooth weight function, the details of which are given in Definition \ref{weight-function} below. To be precise, all implicit constants are allowed to depend on the $L^\infty$-norms of $w$ and all its derivatives.

When splitting summation ranges dyadically, we will write
\[\sum_{\substack{L\leq N/2,\\\mathrm{dyadic}}}\ldots,\]
when the summation over $L$ is to be over the values $N/2$, $N/4$, $N/8$, \dots, where $N$ is a positive real. These sums will always be finite because the summands will vanish identically for small $L$. Analogous notation will also be used for various subsums of dyadic sums.

When $h\in\mathbb Z$ and $k\in\mathbb Z_+$ are coprime, then $\overline h$ denotes an integer such that $h\overline h\equiv1\pmod k$.

\section{The Results}

Let us fix a holomorphic cusp form $F$ of an even weight $\kappa\in\mathbb Z_+$ for the full modular group $\mathrm{SL}(2,\mathbb Z)$. Then $F$ has a Fourier expansion
\[F(z)=\sum_{n=1}^\infty a(n)\,n^{(\kappa-1)/2}\,e(nz),\]
where, as usual, $z\in\mathbb C$ with $\Im z>0$. Our main theorem on fourth moments is as follows.
\begin{theorem}\label{fourth-moment-of-short-exponential-sums}
Let $M\in\left[1,\infty\right[$, let $\Delta\in\left[1,M\right]$, and let $h\in\mathbb Z$ and $k\in\mathbb Z_+$ be coprime. If $k\ll M^{-1/2}\,\Delta$ and $k\ll M^{1/4}$, then
\[
\int\limits_M^{2M}\,\left|\sum_{x\leq n\leq x+\Delta}a(n)\,e\!\left(n\,\frac{h}{k}\right)\right|^4\mathrm dx\ll k^2\,M^{2+\varepsilon}.
\]
If $k\gg M^{-1/2}\,\Delta$ and $k\ll M^{-1/4}\,\Delta^{2/3}$, then we have
\[\int\limits_M^{2M}\,\left|\sum_{x\leq n\leq x+\Delta}a(n)\,e\!\left(n\,\frac{h}{k}\right)\right|^4\mathrm dx\ll M^{1+\varepsilon}\,\Delta^2.\]
\end{theorem}
\noindent
In the proof we shall mostly, but not entirely, argue analogously to the proof of Theorem 4 in \cite{Ivic2009}. 

We wish to detect cancellation in higher moments of short exponential sums. For this purpose, we will estimate the rarity of large values of short exponential sums as follows.
\begin{theorem}\label{large-values}
Let $M,V\in\left[1,\infty\right[$, let $\Delta\in\left[1,M\right]$, let $\delta\in\mathbb R_+$ be fixed, and let $h$ and $k$ be coprime integers with $1\leq k\leq M$, and assume that $k\,M^{2\delta}\ll V\ll k\,M^{1/2+\delta}$. Let $x_1,x_2,\ldots,x_R\in\left[M,2M\right]$, where $R\in\mathbb Z_+$, and assume that $\left|x_i-x_j\right|\geq V$ for $i,j\in\left\{1,2,\ldots,R\right\}$ with $i\neq j$. Fix an exponent pair $\left\langle p,q\right\rangle\in\left]0,1/2\right]\times\left[1/2,1\right]$. If
\[\sum_{x_i\leq n\leq x_i+\Delta}a(n)\,e\!\left(n\,\frac{h}k\right)\gg V\]
for each $i\in\left\{1,2,\ldots,R\right\}$, and
$k^{2/3}\,\Delta^{2/3}\,M^{-1/3+\delta}\ll V\ll\Delta\,M^\delta$,
then
\[R\ll k^2\,M^{1+7\delta}\,\Delta^2\,V^{-5}
+k^{2q/p}\,\Delta^{2+2/p}\,M^{1+q/p+\delta\left(6+5/p+2q/p\right)}\,V^{-2q/p-4-3/p}.\]
\end{theorem}

\begin{remark*}
Naturally, the sums in question are always $\ll\Delta\,M^\varepsilon$ and $\ll\sqrt M$, so the condition $V\ll k\,M^{1/2+\varepsilon}$ will certainly be satisfied in any reasonable application of the result.
\end{remark*}

\begin{theorem}\label{moments-from-large-values}
Let $M\in\left[1,\infty\right[$ and $\Delta\in\left[1,M\right]$, and let $h$ and $k$ be coprime integers with $1\leq k\leq M$.
Let $\alpha,\beta,\gamma\in\mathbb R$ be fixed so that
\[\sum_{x\leq n\leq x+\Delta}a(n)\,e\!\left(n\,\frac{h}k\right)\ll k^\alpha\,\Delta^\beta\,M^\gamma\]
for $x\in\left[M,2M\right]$.
Let $V_0\in\left[1,\infty\right[$ be a parameter such that $k\ll V_0\ll k^\alpha\,\Delta^\beta\,M^\gamma$ and $V_0\gg k^{2/3}\,\Delta^{2/3}\,M^{-1/3}$.
Also, let $A\in\left[2,\infty\right[$ be fixed and let $\left\langle p,q\right\rangle\in\left]0,1/2\right]\times\left[1/2,1\right]$ be a fixed exponent pair. Then,
\[\int\limits_M^{2M}\,\left|\sum_{x\leq n\leq x+\Delta}a(n)\,e\!\left(n\,\frac{h}k\right)\right|^A\mathrm dx\ll M^{1+\varepsilon}\,V_0^A+\Phi+\Psi,\]
where
\[\Phi=\left\{\!\!\begin{array}{ll}
k^2\,M^{1+\varepsilon}\,\Delta^2\,V_0^{A-4}&\text{if $A\leq4$, and}\\
k^{\alpha A-4\alpha+2}\,\Delta^{\beta A-4\beta+2}\,M^{\gamma A-4\gamma+1+\varepsilon}&\text{if $A\geq4$,}
\end{array}\right.\]
and
\[\Psi=\begin{cases}
k^{2q/p}\,\Delta^{2+2/p}\,M^{1+q/p+\varepsilon}\,V_0^{A-2q/p-3-3/p}\quad\text{if $A\leq2q/p+3+3/p$,}\\
k^{2q/p}\,\Delta^{2+2/p}\,M^{1+q/p+\varepsilon}\left(k^\alpha\,\Delta^\beta\,M^\gamma\right)^{A-2q/p-3-3/p}\quad\text{otherwise.}
\end{cases}\]
\end{theorem}

It is of course not immediately obvious what exactly this implies. Possible interesting choices for $\left\langle\alpha,\beta,\gamma\right\rangle$ are the estimate via absolute values $\left\langle0,1,\varepsilon\right\rangle$ made possible by Deligne's famous work \cite{Deligne1974}, the estimate for short sums $\left\langle0,1/6,1/3+\varepsilon\right\rangle$ due to \cite[Theorem~5.5]{Ernvall-Hytonen--Karppinen2008} which holds when $\Delta\ll M^{2/3}$ \cite[Theorem~3]{Jaasaari--Vesalainen2015}, the classical pointwise estimate $\left\langle2/3,0,1/3+\varepsilon\right\rangle$ \cite[Corollary on p.~30]{Jutila1987a}, as well as the improved pointwise bound $\left\langle1/4,0,3/8+\varepsilon\right\rangle$ which holds for $M^{1/10}\ll k\ll M^{1/4}$ \cite[Theorem~1]{Vesalainen2014} as well as for $M^{1/4}\ll k\ll M^{5/18}$ \cite[Corollary~5]{Jaasaari--Vesalainen2015}.

Let us consider as an example large values of $A$ and $k$ with sums of length $\Delta\asymp M^{5/12}$. We obtain the following upper bound.
\begin{theorem}\label{example-bound1}
Let $M\in\left[1,\infty\right[$ and let $\Delta\in\left[1,\infty\right[$ with  $\Delta\asymp M^{5/12}$. Furthermore,  let $A\in\left[11,\infty\right[$ be fixed, and let $h$ and $k$ be coprime integers with $k$ positive and assume that $M^{1/9}\ll k\ll M^{7/18}$. Then
\[
\int\limits_M^{2M}\left|\sum_{x\leq n\leq x+\Delta}a(n)\,e\!\left(n\,\frac{h}{k}\right)\right|^A dx\ll k^2\,M^{11/6+29(A-4)/72+\varepsilon}.
\]
\end{theorem}

As another example, let us consider moments with $A\leq11$ and $k=1$ of sums of length $\ll M^{4/9}$.
In the following theorem, some of the ranges are treated using similar moment results for long sums from \cite{Vesalainen2014}.
\begin{theorem}\label{sovellus2}
Let $M,\Delta\in\left[1,\infty\right[$ with $M^{1/5}\ll\Delta \ll M^{4/9}$ and let $A\in\left[4,11\right]$ be fixed. Then we have
\begin{multline*}
\int\limits_M^{2M}\left|\sum_{x\leq n\leq x+\Delta}a(n)\right|^Adx
\\
\ll \begin{cases} M^{A/11+1+\varepsilon}\,\Delta^{6A/11} & \text{when }A\leq 8\ \text{and}\ \Delta \ll M^{7/24},\\
M^{A/4+1+\varepsilon} & \text{when }A\leq 8\ \text{and}\ \Delta\gg M^{7/24},\\
M^{A/11+1+\varepsilon}\,\Delta^{6A/11} & \text{when }A\geq 8\ \text{and}\ \Delta\ll M^{4/9-11/(9A)},\\
 M^{(A+1)/3+\varepsilon} & \text{when }A\geq 8\ \text{and }\Delta \gg M^{4/9-11/(9A)}.\end{cases}\end{multline*}
\end{theorem}

\section{Some useful theorems, lemmas and corollaries}

\subsection{The truncated Voronoi identity for cusp forms}

As is to be expected, the proofs use a truncated Voronoi type identity for cusp forms. The following is contained in Theorem 1.1 in \cite{Jutila1987a}.
\begin{theorem}\label{truncated-voronoi-identity-for-sums}
Let $x\in\left[1,\infty\right[$ and $N\in\mathbb R_+$ with $1\ll N\ll x$, and let $h$ and $k$ be coprime integers such that $1\leq k\leq x$. Then
\begin{multline*}
\sum_{n\leq x}a(n)\,e\!\left(n\,\frac{h}{k}\right)=\frac{k^{1/2}\,x^{1/4}}{\pi\,\sqrt2}\sum_{n\leq N}a(n)\,e\!\left(-n\,\frac{\bar{h}}{k}\right)n^{-3/4}\,\cos\!\left(4\pi\,\frac{\sqrt{nx}}k-\frac\pi4\right)\\+O(k\,x^{1/2+\varepsilon}\,N^{-1/2}).\end{multline*}
\end{theorem}
Strictly speaking, Theorem 1.1 in \cite{Jutila1987a} assumes that $N\geq1$ instead of $N\gg1$. However, if $N\in\left[c,1\right[$, where $c\in\left]0,1\right[$ is fixed, then the identity still holds as stated, for the left-hand side is $\ll x^{1/2}$ by the Wilton--Jutila estimate, and the right-hand side reduces to the $O$-term $O(k\,x^{1/2+\varepsilon})$.

\subsection{Spacing of square roots}

The truncated Voronoi identity leads to exponential sums involving cusp form coefficients with square root phase factors. When expanding a fourth power of such sums we obtain summation over quadruples $\left\langle a,b,c,d\right\rangle$. Individual terms will have phase factors involving $\sqrt{\vphantom ba}+\sqrt b-\sqrt{\vphantom bc}-\sqrt d$, and so we will need a result on the spacing of square roots. The following is contained in Theorem 2 of \cite{Robert--Sargos2006}.
\begin{theorem}\label{robert-sargos}
Let $\omega\in\left]1,\infty\right[$ be fixed, let $\delta\in\mathbb R_+$ and let $L\geq2$ be an integer. Then the number of quadruples $\left\langle a,b,c,d\right\rangle$ of integers with $a,b,c,d\in\left]L,2L\right]$, and
\[\bigl|a^{1/\omega}+b^{1/\omega}-c^{1/\omega}-d^{1/\omega}\bigr|<\delta\,L^{1/\omega},\]
is
\[\ll \delta\,L^{4+\varepsilon}+L^{2+\varepsilon}.\]
\end{theorem}
\noindent
We shall actually use the special case $\omega=2$ and $\delta=k\, M^{\varepsilon-1/2}\,L^{-1/2}$ for some $M\in\left[1,\infty\right[$ and $k\in\mathbb Z_+$ with $L\ll M$. It is convenient to observe that in fact real values $L\in\left[2,\infty\right[$ are admissible, for if $L$ is not an integer, then we may apply Theorem \ref{robert-sargos} with $\left\lfloor L\right\rfloor$ and $\left\lceil L\right\rceil$. The quadruples not covered by these two cases must feature $\left\lfloor L\right\rfloor+1$ and at least one of the two numbers $2\left\lceil L\right\rceil-1$ and $2\left\lceil L\right\rceil$, so that there are at most $\ll L^2$ such quadruples. Finally, when $L$ is smaller, say $L\in\left[1/2,2\right]$, then the number of quadruples is certainly $\ll L^4\ll1\ll L^2$. Thus we have access to the following corollary.
\begin{corollary}\label{spacing-corollary}
Let $M\in\left[1,\infty\right[$, $L\in\left[1/2,\infty\right[$, $\vartheta\in\mathbb R_+$ and $k\in\mathbb Z_+$. Then the number of quadruples $\left\langle a,b,c,d\right\rangle$ of integers with $a,b,c,d\in\left]L,2L\right]$ and
\[\bigl|\sqrt{\vphantom ba}+\sqrt b-\sqrt{\vphantom bc}-\sqrt d\bigr|<k\,M^{\vartheta-1/2}\]
is
\[\ll L^{7/2+\varepsilon}\,k\,M^{\vartheta-1/2}+L^{2+\varepsilon}.\]
\end{corollary}

\subsection{Plain exponential sums}

Our large value estimate depends on estimating certain plain exponential sums. We will do so by employing the machinery of exponent pairs. If $\left\langle p,q\right\rangle\in\left[0,1/2\right]\times\left[1/2,1\right]$ is known to be an exponent pair, then
\[\sum_{M\leq n\leq M+\Delta}e(A\sqrt n)\ll A^p\,M^{q-p/2}+A^{-1}\,M^{1/2},\]
for $M\in\left[1,\infty\right[$, $\Delta\in\left[1,M\right]$, and $A\in\mathbb R_+$. A good reference for the theory of exponent pairs is \cite{Graham--Kolesnik1991}.

We also need the following result which allows us to separate Fourier coefficients from the exponential sums. It is a lemma of Bombieri, and appears as Lemma 1.5 in \cite{Montgomery1971}.
\begin{theorem}
Let $H$ be a complex Hilbert space with inner product $\left\langle\cdot\middle|\cdot\right\rangle$ and norm $\left\|\cdot\right\|$. Also, let $\xi,\varphi_1,\varphi_2,\dots,\varphi_R\in H$, where $R\in\mathbb Z_+$. Then
\[\sum_{r=1}^R\left|\left\langle\xi\middle|\varphi_r\right\rangle\right|^2
\leq\left\|\xi\right\|^2\max_{1\leq r\leq R}\sum_{s=1}^R\left|\left\langle\varphi_r\middle|\varphi_s\right\rangle\right|.\]
\end{theorem}
\noindent We shall apply this theorem with $H=\mathbb C^N$ for some $N\in\mathbb Z_+$ with the usual inner product and norm, which for vectors $z=\left\langle z_1,\ldots,z_N\right\rangle,w=\left\langle w_1,\ldots,w_N\right\rangle\in\mathbb C^N$ are given by
\[\left\langle z\middle|w\right\rangle=\sum_{\ell=1}^N\overline{z_\ell}\,w_\ell
\quad\text{and}\quad\left\|z\right\|^2=\sum_{\ell=1}^N\left|z_\ell\right|^2.\]

\subsection{Exponential integrals}

We will need a lemma for estimating exponential integrals. The following is Lemma 6 in \cite{Jutila--Motohashi2005}.
\begin{lemma}\label{jutila-motohashi-lemma}
Let $a,b\in\mathbb R_+$ and $a<b$, let $g\in C_{\mathrm c}^\infty(\mathbb R_+)$ with $\mathrm{supp}\,g\subseteq\left[a,b\right]$, and let $G_0,G_1\in\mathbb R_+$ be such that
\[g^{(\nu)}(x)\ll_\nu G_0\,G_1^{-\nu}\]
for all $x\in\mathbb R_+$ for each nonnegative integer $\nu$. Also, let $f$ be a holomorphic function defined in $D\subset\mathbb C$, which consists of all points in the complex plane with distance smaller than $\rho\in\mathbb R_+$ from the interval $\left[a,b\right]$ of the real axis. Assume that $f$ is real-valued on $\left[a,b\right]$ and let $F_1\in\mathbb R_+$ be such that
\[\left|f'(z)\right|\gg F_1\]
for all $z\in D$. Then, for all positive integers $P$,
\[\int\limits_a^bg(x)\,e(f(x))\,\mathrm dx\ll_PG_0\left(G_1\,F_1\right)^{-P}\left(1+\frac{G_1}\rho\right)^P\left(b-a\right).\]
\end{lemma}

We remark that, when $f$ is holomorphic in $\left\{z\in\mathbb C\middle|\Re z>0\right\}$, we may choose $\rho$ so that $\rho\asymp a$. In particular, in our applications of the lemma, we have $a\asymp b\asymp G_1$ and the factor $\left(1+G_1/\rho\right)^P$ is always~$\ll_P1$.

In the proof of the fourth moment estimate, we will introduce to our integrals a smooth weight function $w$. For definiteness, we define it here:
\begin{definition}\label{weight-function}
In the following, $w$ will denote a function in $C_{\mathrm c}^\infty(\mathbb R_+)$, depending on $M\in\left[1,\infty\right[$, taking values only from the interval $\left[0,1\right]$, and satisfying $\mathrm{supp}\,w\subseteq\left[M/2,5M/2\right]$, $w\equiv1$ on $\left[M,2M\right]$, and
\[w^{(\nu)}(x)\ll_\nu M^{-\nu}\]
for all $x\in\mathbb R_+$, for every $\nu\in\mathbb Z_+\cup\left\{0\right\}$.
\end{definition}

The following lemma will be used to estimate several exponential integrals:
\begin{lemma}\label{sinitulo}
Let $n,k\in\mathbb Z_+$, let $\Delta\in\mathbb R_+$, assume that $n\ll k^2\,M\,\Delta^{-2}$, and write
\[
S(n)=\sin\left(2\pi\,\frac{\sqrt{n}}{k}\left(\sqrt{x+\Delta}-\sqrt{x}\right)\right).
\]
Also, let $M$ and $w(x)$ be as in Definition~\ref{weight-function}.
Then
\[
\frac{\partial^{\nu}}{\partial x^{\nu}}\left(w(x)\,S(a)\,S(b)\,S(c)\,S(d)\right)\ll_\nu \frac{(abcd)^{1/2}\,\Delta^4}{k^4\,M^{2+\nu}}
\]
for all $a,b,c,d\in\mathbb Z_+$ with $\max(a,b,c,d)\ll k^2\,M\,\Delta^{-2}$ and $x\in\mathbb R_+$, for every $\nu\in\mathbb Z_+\cup\left\{0\right\}$.
\end{lemma}

\begin{proof} Since $w(x)$ vanishes outside the interval $\left[M/2,5M/2\right]$, it is enough to consider the case $x\in\left[M/2,5M/2\right]$. Notice first that we have
\begin{multline*}
\frac{\partial^{\nu}}{\partial x^{\nu}}\left(w(x)\,S(a)\,S(b)\,S(c)\,S(d)\right)
=\sum_{\substack{\alpha_1+\alpha_2+\alpha_3\\+\alpha_4+\alpha_5=\nu}}\frac{\nu!}{\alpha_1!\,\alpha_2!\,\alpha_3!\,\alpha_4!\,\alpha_5!}\\
\cdot\left(\frac{\partial^{\alpha_1}}{\partial x^{\alpha_1}}\,S(a)\right)\left(\frac{\partial^{\alpha_2}}{\partial x^{\alpha_2}}\,S(b)\right)\left(\frac{\partial^{\alpha_3}}{\partial x^{\alpha_3}}\,S(c)\right)\left(\frac{\partial^{\alpha_4}}{\partial x^{\alpha_4}}\,S(d)\right)\left(\frac{\partial^{\alpha_5}}{\partial x^{\alpha_5}}\,w(x)\right),
\end{multline*}
where the summation is over quintuples $\left\langle\alpha_1,\alpha_2,\alpha_3,\alpha_4,\alpha_5\right\rangle$ of nonnegative integers satisfying $\alpha_1+\alpha_2+\alpha_3+\alpha_4+\alpha_5=\nu$.
Now
\[
S(n)=\sin\left(2\pi \frac{\sqrt{n}}{k}(\sqrt{x+\Delta}-\sqrt{x})\right)\ll \frac{\sqrt{n}}{k}(\sqrt{x+\Delta}-\sqrt{x})\ll\frac{\sqrt{n}\,\Delta}{k\,\sqrt{x}},
\]
and when $\alpha\in\mathbb Z_+\cup\left\{0\right\}$, we have
\[
\frac{\partial^{\alpha}}{\partial x^{\alpha}}\,S(n)
\ll_\alpha\frac{\sqrt{n}\,\Delta}{k\,x^{(2\alpha+1)/2}}.
\]
Putting everything together, we obtain
\begin{multline*}
\frac{\partial^{\nu}}{\partial x^{\nu}}\left(w(x)\,S(a)\,S(b)\,S(c)\,S(d)\right)
=\sum_{\substack{\alpha_1+\alpha_2+\alpha_3\\+\alpha_4+\alpha_5=\nu}}
\frac{\nu!}{\alpha_1!\,\alpha_2!\,\alpha_3!\,\alpha_4!\,\alpha_5!}\\
\cdot\left(\frac{\partial^{\alpha_1}}{\partial x^{\alpha_1}}\,S(a)\right)\left(\frac{\partial^{\alpha_2}}{\partial x^{\alpha_2}}\,S(b)\right)\left(\frac{\partial^{\alpha_3}}{\partial x^{\alpha_3}}\,S(c)\right)\left(\frac{\partial^{\alpha_4}}{\partial x^{\alpha_4}}\,S(d)\right)\left(\frac{\partial^{\alpha_5}}{\partial x^{\alpha_5}}\,w(x)\right)\\
\ll_{\nu} \sum_{\substack{\alpha_1+\alpha_2+\alpha_3\\+\alpha_4+\alpha_5=\nu}}
\frac{\sqrt{\vphantom ba}\,\Delta}{k\,x^{(2\alpha_1+1)/2}}
\cdot\frac{\sqrt{b}\,\Delta}{k\,x^{(2\alpha_2+1)/2}}
\cdot\frac{\sqrt{\vphantom bc}\,\Delta}{k\,x^{(2\alpha_3+1)/2}}
\cdot\frac{\sqrt{d}\,\Delta}{k\,x^{(2\alpha_4+1)/2}}\cdot x^{-\alpha_5}\\
\ll_{\nu}\frac{(abcd)^{1/2}\,\Delta^4}{k^4\,M^{2+\nu}}.
\end{multline*}
\end{proof}

Before embarking on the proofs of Theorems \ref{fourth-moment-of-short-exponential-sums}, \ref{large-values} and \ref{moments-from-large-values}, we introduce one final lemma on the mean square of the kind of exponential sums which arise from the truncated Voronoi identity.
\begin{lemma}\label{second-moment-pienet-termit}
Let $M\in\left[1,\infty\right[$, $L\in\left[1/2,\infty\right[$ and $T\in\left[0,\infty\right[$ with $T\ll M$
and let $h\in\mathbb Z$ and $k\in\mathbb Z_+$ be coprime.
Furthermore, let $w(x)$ be a smooth weight function as in Definition \ref{weight-function}.
Then we have
\begin{multline*}
\int\limits_{M/2}^{5M/2}w(x)\left|\sum_{L<n\leq2L}a(n)\,n^{-3/4}\,e\!\left(-n\,\frac{\overline h}k\right)e\!\left(\pm\frac{2\sqrt{n(x+T)}}k\right)\right|^2\mathrm dx\\
\ll M\,L^{-1/2}+L^\varepsilon\,k\,M^{1/2+\varepsilon},
\end{multline*}
and if we further assume that $L\ll M^{1-\vartheta}\,k^{-2}$ for some fixed positive real number $\vartheta$ that can be chosen to be arbitrarily small, then we have
\begin{multline*}\int\limits_{M/2}^{5M/2}w(x)\left|\sum_{L<n\leq2L}a(n)\,n^{-3/4}\,e\!\left(-n\,\frac{\overline h}k\right)e\!\left(\pm\frac{2\sqrt{n(x+T)}}k\right)\right|^2\mathrm dx\\
\ll M\,L^{-1/2}.
\end{multline*}
\end{lemma}

\begin{proof}
Here we expand the square as $\left|\Sigma\right|^2=\Sigma\,\overline\Sigma$ and separate the diagonal terms from the off-diagonal terms, leading to
\begin{align*}
&\int\limits_{M/2}^{5M/2}w(x)\left|\sum_{L<n\leq2L}a(n)\,n^{-3/4}\,e\!\left(-n\,\frac{\overline h}k\right)e\!\left(\pm\frac{2\sqrt{n(x+T)}}k\right)\right|^2\mathrm dx\\
&\ll\sum_{L<n\leq2L}\frac{\left|a(n)\right|^2}{n^{3/2}}\int\limits_{M/2}^{5M/2}w(x)\,\mathrm dx\\
&\qquad+\sum_{L<m<n\leq2L}\frac{\left|a(m)\,a(n)\right|}{(mn)^{3/4}}
\left|\int\limits_{M/2}^{5M/2}w(x)\,e\!\left(\pm\frac{2\left(\sqrt m-\sqrt n\right)\sqrt{x+T}}k\right)\mathrm dx\right|.
\end{align*}
The diagonal terms contribute $\ll M\,L^{-1/2}$, and by Lemma \ref{jutila-motohashi-lemma}, the off-diagonal terms contribute, for arbitrary $P\in\mathbb Z_+$,
\[\ll_P\sum_{L<m<n\leq2L}\frac{\left|a(m)\,a(n)\right|}{(mn)^{3/4}}\cdot\left(\frac{k\,M^{-1/2}}{\sqrt n-\sqrt m}\right)^P\cdot M.\]
When ${k}\,{M^{-1/2}\,|\sqrt{n}-\sqrt{m}|^{-1}}\ll M^{-\varepsilon'}$, for some constant $\varepsilon'\in\mathbb R_+$, the bound above can be made as small as desired by choosing $P$ to be sufficiently large (depending on $\varepsilon'$). Let us now choose $\varepsilon'\in\left]0,{\vartheta}/{2}\right[$. The condition
\[
\frac{k}{M^{1/2}\,|\sqrt{n}-\sqrt{m}|}\ll M^{-\varepsilon'}
\]
holds when $|n-m|\gg k\,M^{-1/2+\varepsilon'}\sqrt{L}$. Therefore, when $L\ll M^{1-\vartheta}\,k^{-2}$, we have
\[
k\,M^{-1/2+\varepsilon'}\sqrt{L}\ll k^{1-1}\,M^{-1/2+\varepsilon'+1/2-\vartheta/2}=o(1),
\]
and hence, when $m\ne n$ and so $\left|m-n\right|\geqslant1$, we have
\[
\frac{k}{M^{1/2}\,|\sqrt{n}-\sqrt{m}|}\ll M^{-\varepsilon'},
\]
and thus, in particular, only the contribution from diagonal terms counts when $L\ll M^{1-\vartheta}\,k^{-2}$.

Let us now estimate the contribution coming from the off-diagonal terms for which $\left|n-m\right|\ll k\,M^{-1/2+\varepsilon'}\,\sqrt L$, when $L\gg M^{1-\vartheta}\,k^{-2}$. For each value of $n$, there are $\ll \sqrt{L}\,k\,M^{-1/2+\varepsilon'}$ values of $m$, and for all of these values, we estimate the integral by absolute values. We thus obtain from the off-diagonal terms
\[
\ll \sum_{L< n\leq 2L}L^{\varepsilon-3/2}\,\sqrt{L}\,k\,M^{-1/2+\varepsilon}M\ll L^\varepsilon\,k\,M^{1/2+\varepsilon'}.
\]
\end{proof}

\subsection{Moments of long linear sums}

The following is Theorem 2.3 from \cite{Vesalainen2014}.

\begin{theorem}\label{esa_moments}
Let $M\in\left[1,\infty\right[$,
let us fix an exponent pair $\left\langle p,q\right\rangle\in\left]0,1/2\right]\times\left[1/2,1\right]$ satisfying
$q\geq(p+1)/2$, and let $h$ and $k$ be coprime integers with
$1\leq k\ll M^{1/2-\varepsilon}$.
Furthermore, let
$\alpha,\beta,\gamma,\delta,A\in\left[0,\infty\right[$ be fixed exponents so
that
\[\sum_{n\leq x}a(n)\,e\!\left(n\,\frac{h}k\right)\ll
k^\alpha\,x^{\beta+\varepsilon}\]
for $x\in\left[1,\infty\right[$ and for $k$ satisfying $x^\gamma\ll k\ll x^\delta$.
Then, for $M^\gamma\ll k\ll M^\delta$,
\[\int\limits_M^{2M}\left|\sum_{n\leq x}a(n)\,e\!\left(n\,\frac hk\right)\right|^A\,\mathrm dx
\ll k^{A/2}\,M^{A/4+1}+\Phi+\Psi,\]
where
\[\Phi=\left\{\begin{array}{ll}
k^{\alpha A+2(1-\alpha)}\,M^{\beta A+(1-2\beta)+\varepsilon}&\text{if
$A\geq2$},\\
k^{A/2+1}\,M^{A/4+1/2+\varepsilon}&\text{if $A\leq2$},
\end{array}\right.\]
and
\[\Psi=\left\{\begin{array}{ll}
k^{\alpha A-\alpha-\alpha/p+(1-\alpha)2q/p}\,M^{\beta
A+1-\beta-\beta/p+(1-2\beta)q/p+\varepsilon}&
\text{if $A\geq1+\left(1+2q\right)/p$,}\\[1mm]
k^{A/2-1/2-1/(2p)+q/p}\,M^{A/4+3/4-1/(4p)+q/(2p)+\varepsilon}&\text{if
$A\leq1+\left(1+2q\right)/p$.}
\end{array}\right.\]
\end{theorem}

\section{Proof of Theorem \ref{fourth-moment-of-short-exponential-sums}}\label{fourth-moment-proof-section}

We let $\varepsilon_0\in\mathbb R_+$ be arbitrary. Our goal is to prove an estimate $\ll M^{2+\varepsilon_0}\,k^2$ or $\ll M^{1+\varepsilon_0}\,\Delta^2$. Some exponents in the proof will depend on the desired final value of $\varepsilon_0$. We assume throughout the proof that $k\ll M^{-1/2}\,\Delta$ and $k\ll M^{1/4}$, or that $k\gg M^{-1/2}\,\Delta$ and $k\ll M^{-1/4}\,\Delta^{2/3}$.

We begin by applying the truncated Voronoi identity for cusp form coefficients  to get, for $N\in\mathbb R_+$ satisfying $1\ll N\ll M$,
\begin{align*}
&\int\limits_M^{2M}\,\left|\sum_{x\leq n\leq x+\Delta}a(n)\,e\!\left(n\,\frac{h}{k}\right)\right|^4\mathrm dx\\
&\ll\int\limits_M^{2M}\,\Biggl|k^{1/2}\sum_{n\leq N}a(n)\,n^{-3/4}\,e\!\left(-n\,\frac{\overline{h}}{k}\right)\Biggl((x+\Delta)^{1/4}\,\cos\biggl(4\pi\,\frac{\sqrt{n(x+\Delta)}}{k}-\frac\pi4\biggr)\Biggr.\Biggr.\\
&\hspace*{12em}\Biggl.\Biggl.-x^{1/4}\,\cos\!\left(4\pi\,\frac{\sqrt{nx}}{k}-\frac\pi4\right)\!\Biggr)\Biggr|^4\mathrm dx
+M^{3+\varepsilon}\,N^{-2}\,k^{4},
\end{align*}
where we applied the elementary inequality $\left|A+B\right|^4\ll\left|A\right|^4+\left|B\right|^4$, which holds uniformly for all $A,B\in\mathbb C$.

Let us now choose $N$ in the following way:
\[
N=\begin{cases} M^{1/2}\,k, & \textrm{when } k\ll \Delta\, M^{-1/2},\\ k^{2}\,M\,\Delta^{-1}, &\textrm{otherwise}.\end{cases}
\]
Thus, when $k\ll\Delta\,M^{-1/2}$, we have trivially $N\geq1$, and we have $N\ll M$ since $k\ll M^{1/2}$. When $k\gg\Delta\,M^{-1/2}$, we have again trivially $N\gg1$, and we have $N\ll M$ since $k\ll M^{-1/4}\,\Delta^{2/3}\ll\Delta^{2/3-1/4}=\Delta^{5/12}\ll\Delta^{1/2}$.

Hence the error from the error term of the truncated Voronoi identity becomes
\[
M^{3+\varepsilon}\,N^{-2}\,k^4\ll \begin{cases} M^{2+\varepsilon}\,k^2,& \textrm{when }k\ll \Delta \,M^{-1/2},\\ M^{1+\varepsilon}\,\Delta^2, & \textrm{otherwise.}\end{cases}
\]

Since the integrand is nonnegative, we may introduce the weight function $w(x)$ of Definition \ref{weight-function} to the integral involving the main terms from the truncated Voronoi identity, and extend the region of integration to be over the interval $\left[M/2,5M/2\right]$:
\[\int\limits_M^{2M}\left|\ldots\right|^4\mathrm dx
\ll\int\limits_{M/2}^{5M/2}w(x)\left|\ldots\right|^4\mathrm dx.\]
Next, we split the sum $\sum_n$ dyadically into $\sum_{L\leq N/2}\sum_{L<n\leq2L}$, where $L$ ranges over the values $N/2$, $N/4$, $N/8$, \dots\ There will be $\ll1+\log M$ such values of interest, and so we may continue the estimations by applying H\"older's inequality to get
\begin{multline*}
\ll \left(1+\log M\right)^3\sum_{\substack{L\leq N/2\\\text{dyadic}}}\int\limits_{M/2}^{5M/2}w(x)\\\cdot\left|k^{1/2}\sum_{L<n\leq2L}a(n)\,n^{-3/4}\,e\!\left(-n\,\frac{\overline h}k\right)\left((x+\Delta)^{1/4}\cos(\ldots)-x^{1/4}\cos(\ldots)\right)\right|^4\mathrm dx,
\end{multline*}
where of course $\left(1+\log M\right)^3\ll M^\varepsilon$.
The sum over $L$ is split into three parts: those terms with $L$ large, the terms with $L$ so small, that there is very little oscillation, but there is cancellation in the main terms of the truncated Voronoi identity, and the remaining terms in the middle:
\[\sum_{\substack{L\leq N/2\\\textrm{dyadic}}}=
\sum_{\substack{L\ll Y\\\textrm{dyadic}}}
+\sum_{\substack{Y\ll L\ll X\\\textrm{dyadic}}}
+\sum_{\substack{X\ll L\leq N/2\\\textrm{dyadic}}}.\]

When $k\ll\Delta\,M^{-1/2}$ we choose
\[Y=1\quad\text{and}\quad X=\sqrt M,\]
and certainly $1\ll Y\ll X\ll N$.
In particular, the first sum over $L\ll Y$ will be empty.

When $k\gg\Delta\,M^{-1/2}$, we choose
\[Y= k^2\,M\,\Delta^{-2}\quad\text{and}\quad X=\min\left(\Delta^4\,M^{-1}\,k^{-6},N\right).\]
We will trivially have $1\ll Y\ll N$. Also, we always have $Y\ll \Delta^4\,M^{-1}\,k^{-6}$ since this is equivalent with $k\ll M^{-1/4}\,\Delta^{3/4}$ and this holds since we have  $k\ll M^{-1/4}\,\Delta^{2/3}\ll M^{-1/4}\,\Delta^{3/4}$.
When $\Delta \gg M^{2/5}\,k^{8/5}$, we have $X\asymp N$ and the sum over $L\gg X$ will be empty. The rest of the proof consists of working through each of these cases separately.

\paragraph{The high-frequency terms $L\gg X$.} For these values of $L$, we may estimate by the truncated Voronoi identity that for any $x\in\left[M/2,5M/2\right]$ and each $T\in\left\{0,\Delta\right\}$,
\begin{align*}
&k^{1/2}\,(x+T)^{1/4}\sum_{L<n\leq2L}a(n)\,n^{-3/4}\,e\!\left(-n\,\frac{\overline h}k\right)\cos\!\left(4\pi\,\frac{\sqrt n}k\sqrt{x+T}-\frac\pi4\right)\\
&=k^{1/2}\left(x+T\right)^{1/4}\sum_{n\leq2L}\ldots-k^{1/2}\left(x+T\right)^{1/4}\sum_{n\leq L}\ldots\\
&=\pi\,\sqrt2\sum_{n\leq x+T}a(n)\,e\!\left(n\,\frac hk\right)+O(k\,M^{1/2+\varepsilon}\,L^{-1/2})\\
&\qquad-\pi\,\sqrt2\sum_{n\leq x+T}a(n)\,e\!\left(n\,\frac hk\right)-O(k\,M^{1/2+\varepsilon}\,L^{-1/2})
\ll k\,M^{1/2+\varepsilon}\,L^{-1/2},
\end{align*}
so that
\begin{multline*}
\left|k^{1/2}\sum_{L<n\leq2L}a(n)\,n^{-3/4}\,e\!\left(-n\,\frac{\overline h}k\right)\left((x+\Delta)^{1/4}\cos(\ldots)-x^{1/4}\cos(\ldots)\right)\right|^2\\
\ll k^2\,M^{1+\varepsilon}\,L^{-1}.
\end{multline*}
The contribution from the high-frequency terms involving $\sqrt{x+T}$, where $T\in\left\{0,\Delta\right\}$, can be estimated by
\begin{multline*}
\ll k^{3}\,M^{3/2+\varepsilon}\sum_\pm\sum_{\substack{X\ll L\leq N/2\\\text{dyadic}}}L^{-1}\\
\cdot\int\limits_{M/2}^{5M/2}w(x)\left|\sum_{L<n\leq2L}a(n)\,n^{-3/4}\,e\!\left(-n\,\frac{\overline h}k\right)e\!\left(\pm2\,\frac{\sqrt n}k\,\sqrt{x+T}\right)\right|^2\mathrm dx.
\end{multline*}
We may use  Lemma \ref{second-moment-pienet-termit} to bound the expression on the second line. The contribution coming from this is $\ll M\,L^{-1/2}+L^\varepsilon\,k\,M^{1/2+\varepsilon}$, and if $L\ll M^{1-\varepsilon_0/2}\,k^{-2}$, then the contribution is $\ll M\, L^{-1/2}$.

The contribution coming from the diagonal terms $M\,L^{-1/2}$ is
\begin{align*}
&\ll k^3\,M^{3/2+\varepsilon}\sum_{\substack{X\ll L\leq N/2\\ \text{dyadic}}}L^{-1}\,M\,L^{-1/2}
 \ll k^3\,M^{5/2+\varepsilon}\,X^{-3/2}.
\end{align*}
When $k\ll \Delta \,M^{-1/2}$, we have $X=M^{1/2}$, and hence the contribution will be
\[
\ll k^3\,M^{5/2+\varepsilon}\,M^{-3/4}=k^3\,M^{7/4+\varepsilon}\ll k^2\,M^{2+\varepsilon},
\]
since $k\ll M^{1/4}$. When $k\gg \Delta\, M^{-1/2}$, we have $X\asymp N$, in which case there are no high-frequency terms to consider, or $X=\Delta^4\,M^{-1}\,k^{-6}$. In the latter case we obtain
\begin{align*}
&\ll k^3\,M^{5/2+\varepsilon}\,X^{-3/2}\ll k^3\,M^{5/2+\varepsilon}(\Delta^4\,M^{-1}\,k^{-6})^{-3/2}\\
&\ll k^{12}\,M^{4+\varepsilon}\,\Delta^{-6}\ll M^{1+\varepsilon}\,\Delta^2,
\end{align*}
since $k\ll \Delta^{2/3}\,M^{-1/4}$.

Finally, let us compute the contribution of the term $L^\varepsilon\,k\,M^{1/2+\varepsilon}$. This term exists only for $L\gg M^{1-\varepsilon_0/2}\,k^{-2}$, and we thus obtain
\begin{align*}
&\ll k^3\,M^{3/2+\varepsilon}\sum_{\substack{M^{1-\varepsilon_0/2}\,k^{-2}\ll L\leq N/2\\ \textrm{dyadic}}}L^{\varepsilon-1}\,k\,M^{1/2+\varepsilon}\\
&\ll k^4\,M^{2+\varepsilon}\left(M^{1-\varepsilon_0/2}\,k^{-2}\right)^{\varepsilon-1}
\ll M^{1+\varepsilon_0/2+\varepsilon}\,k^6.
\end{align*}
In the case $k\ll\Delta\,M^{-1/2}$ this is $\ll k^2\,M^{2+\varepsilon_0}$ since $k\ll M^{1/4}$. In the case $k\gg\Delta\,M^{-1/2}$ this is $\ll M^{1+\varepsilon_0}\,\Delta^2$, provided that $k\ll\Delta^{1/3}$. But this holds since $M^{-1/2}\,\Delta\ll k\ll M^{-1/4}\,\Delta^{2/3}$, so that $\Delta\ll M^{3/4}$, and therefore $k\ll M^{-1/4}\,\Delta^{2/3}\ll \Delta^{-1/3}\,\Delta^{2/3}\ll\Delta^{1/3}$.

\paragraph{The low-frequency terms $L\ll Y$.}
Let us recall first that these terms need to be considered only in the case $k\gg\Delta\,M^{-1/2}$ in which $Y=k^{2}\,M\,\Delta^{-2}$.

For low-frequency terms we want to get the sums to partially cancel each other, and therefore, we want to replace the factor $(x+\Delta)^{1/4}$ by $x^{1/4}$:
\begin{multline*}
k^{1/2}\sum_{L< n\leq 2L}\frac{a(n)}{n^{3/4}}\,e\!\left(-n\,\frac{\overline{h}}{k}\right)\left((x+\Delta)^{1/4}-x^{1/4}\right)\cos\left(4\pi\,\frac{\sqrt{n(x+\Delta)}}{k}-\frac{\pi}{4}\right)\\
\ll  k^{1/2}\,L^{1/4}\,\Delta M^{-3/4}\ll k^{1/2}\left(k^2\,M\,\Delta^{-2}\right)^{1/4}\Delta \,M^{-3/4}
\ll k\,M^{-1/2}\,\Delta^{1/2}.
\end{multline*}
Hence the total contribution coming from replacing $\left(x+\Delta\right)^{1/4}$ by $x^{1/4}$ is
\[
\ll M^{1+\varepsilon}\left(k\,M^{-1/2}\,\Delta^{1/2}\right)^4
\ll k^4\,M^{-1+\varepsilon}\,\Delta^2\ll \Delta^2\,M^{1+\varepsilon},
\]
which holds since $k\ll M^{-1/4}\,\Delta^{2/3}\ll M^{2/3-1/4}=M^{5/12}\ll M^{1/2}$.
We may now use the elementary trigonometric identity
\[\cos\left(2\xi-\frac\pi4\right)-\cos\left(2\eta-\frac\pi4\right)
=2\sin\left(\xi-\eta\right)\cos\left(\xi+\eta+\frac\pi4\right),\]
which holds for any $\xi,\eta\in\mathbb R$.
Applying this with
\[\xi=2\pi\,\frac{\sqrt n}k\,\sqrt{x+\Delta}\quad\text{and}\quad
\eta=2\pi\,\frac{\sqrt n}k\,\sqrt{x},\]
the contribution from the terms with $L\ll Y$ is
\begin{align*}
&\ll k^2\,M^{1+\varepsilon}\sum_\pm\sum_{\substack{L\ll Y\\\text{dyadic}}}\\
&\quad\cdot
\int\limits_{M/2}^{5M/2}w(x)\left|\sum_{L<n\leq2L}\frac{a(n)}{n^{3/4}}\,e\!\left(-n\,\frac{\overline h}k\right)S(n)\,
e\!\left(\pm\frac{\sqrt n}{k}\,(\sqrt{x+\Delta}+\sqrt x\,)\right)\right|^4\mathrm dx\\
&\ll k^2\,M^{1+\varepsilon}\sum_{\substack{L\ll Y\\\text{dyadic}}}
\sum_{L<a\leq2L}\sum_{L<b\leq2L}\sum_{L<c\leq2L}\sum_{L<d\leq2L}\frac{\left|a(a)\,a(b)\,\overline{a(c)\,a(d)}\right|}{(abcd)^{3/4}}\\
&\qquad\cdot\left|\int\limits_{M/2}^{5M/2}w(x)\,S(a)\,S(b)\,S(c)\,S(d)\,e\!\left(\frac{\alpha}{k}\,(\sqrt{x+\Delta}+\sqrt x\,)\right)\mathrm dx\right|.
\end{align*}
where the factors $S(n)$ are given by
\[S(n)=\sin\left(2\pi\,\frac{\sqrt n}{k}\,(\sqrt{x+\Delta}-\sqrt x)\right),\]
the coefficient $\alpha$ is the square root expression
\[\alpha=\sqrt{\vphantom ba}+\sqrt b-\sqrt{\vphantom bc}-\sqrt d,\]
and $w$ is as in Definition \ref{weight-function}.

Let us first consider the terms of $\sum_a\sum_b\sum_c\sum_d$ with $\alpha\gg M^{\varepsilon_0/2-1/2}\,k$. Using Lemma \ref{sinitulo}, we have, for each $\nu\in\mathbb Z_+\cup\left\{0\right\}$,
\[\frac{\mathrm d^\nu}{\mathrm dx^\nu}\left(w(x)\,S(a)\,S(b)\,S(c)\,S(d)\right)\ll_\nu \frac{(abcd)^{1/2}\,\Delta^4}{k^4\,M^{2+\nu}}.\]
Therefore, in the terms under consideration, the integral $\int_{M/2}^{5M/2}\ldots\mathrm dx$ may be estimated using Lemma \ref{jutila-motohashi-lemma} to be, for any $P\in\mathbb Z_+$,
\begin{multline*}\ll_P \frac{(abcd)^{1/2}\,\Delta^4}{k^4\,M^2}\,\left(M\,\alpha\,M^{-1/2}\,k^{-1}\right)^{-P}\,M
\ll \frac{(abcd)^{1/2}\,\Delta^4}{k^4\,M^2}\,M^{1-P\varepsilon_0/2}\\ \ll (abcd)^{1/2}\,M^{1-P\varepsilon_0/2}.\end{multline*}
Fixing $P$ to be sufficiently large (depending on $\varepsilon_0$), the contribution from the terms under consideration will be
\begin{align*}
&\ll k^2\,M^{1+\varepsilon}\sum_{\substack{L\ll Y\\\text{dyadic}}}\sum_a\sum_b\sum_c\sum_d\frac{\left|a(a)\,a(b)\,a(c)\,a(d)\right|}{(abcd)^{1/4}}\,M^{1-P\varepsilon_0/2}\\
&\ll k^2\, M^{2+\varepsilon-P\varepsilon_0/2}\sum_{\substack{L\ll Y\\\text{dyadic}}}L^3
\ll M^{6+\varepsilon-P\varepsilon_0/2}\ll1.
\end{align*}

Finally, by Corollary \ref{spacing-corollary}, the number of terms in the sum $\sum_a\sum_b\sum_c\sum_d$ with $\alpha\ll k\,M^{\varepsilon_0/2-1/2}$ is
\[\ll L^{7/2+\varepsilon}\,k\,M^{\varepsilon_0/2-1/2}+L^{2+\varepsilon},\]
and so we conclude, estimating everything by absolute values, and sine factors by $\sin x\ll x$, that the rest of the low-frequency terms with $L\ll Y$ contribute
\begin{align*}
&\ll k^2\,M^{1+\varepsilon}\sum_{\substack{L\ll Y\\\text{dyadic}}}\left(L^{7/2+\varepsilon}\,k\,M^{\varepsilon_0/2-1/2}+L^{2+\varepsilon}\right)L^{\varepsilon-3}\,M\left(L^{1/2}\,\Delta\,M^{-1/2}\,k^{-1}\right)^4\\
&\ll k^2\,\Delta^4\,M^\varepsilon\sum_{\substack{L\ll Y\\\text{dyadic}}}\left(L^{5/2+\varepsilon}\,k^{-3}\,M^{\varepsilon_0/2-1/2}+L^{1+\varepsilon}\,k^{-4}\right)\\
&\ll\Delta^4\,M^\varepsilon\left(\left(k^{2}\,M\,\Delta^{-2}\right)^{5/2+\varepsilon}k^{-1}\,M^{\varepsilon_0/2-1/2}+\left(k^{2}\,M\,\Delta^{-2}\right)^{1+\varepsilon}k^{-2}\right)\\
&\ll\Delta^{-1}\,M^{2+\varepsilon_0/2+\varepsilon}\,k^4+\Delta^2\,M^{1+\varepsilon}\ll M^{1+\varepsilon_0}\,\Delta^2,
\end{align*}
since $k\ll \Delta^{2/3}\,M^{-1/4}\ll\Delta^{3/4}\,M^{-1/4}$.

\paragraph{The terms in the middle with $Y\ll L\ll X$.}
The contribution from the terms with $Y\ll L\ll X$ and involving $\sqrt{x+T}$, where $T\in\left\{0,\Delta\right\}$, is
\begin{align*}
&\ll k^2\,M^{1+\varepsilon}\sum_\pm\\
&\qquad\cdot\sum_{\substack{Y\ll L\ll X\\\text{dyadic}}}
\int\limits_{M/2}^{5M/2}w(x)\left|\sum_{L<n\leq2L}a(n)\,n^{-3/4}\,
e\!\left(-n\,\frac{\overline{h}}{k}\pm2\,\frac{\sqrt n}{k}\,\sqrt{x+T}\,\right)\right|^4\mathrm dx\\
&\ll k^2\,M^{1+\varepsilon}\sum_{\substack{Y\ll L\ll X\\\text{dyadic}}}
\sum_{L<a\leq2L}\sum_{L<b\leq2L}\sum_{L<c\leq2L}\sum_{L<d\leq2L}\frac{\bigl|a(a)\,a(b)\,\overline{a(c)\,a(d)}\bigr|}{(abcd)^{3/4}}\\
&\qquad\cdot\left|\int\limits_{M/2}^{5M/2}w(x)\,e\!\left(2\alpha\,\frac{\sqrt{x+T}}{k}\,\right)\mathrm dx\right|,
\end{align*}
where the coefficient $\alpha$ is again the square root expression
\[\alpha=\sqrt{\vphantom ba}+\sqrt b-\sqrt{\vphantom bc}-\sqrt d.\]

In those terms of $\sum_a\sum_b\sum_c\sum_d$ in which $\alpha\gg k\,M^{\varepsilon_0/2-1/2}$, we may estimate the integral $\int_{M/2}^{5M/2}$ by Lemma \ref{jutila-motohashi-lemma} for any $P\in\mathbb Z_+$ by
\[\ll_P(M\,\alpha\,k^{-1}\,M^{-1/2})^{-P}\,M\ll M^{1-P\varepsilon_0/2}.\]
Thus, these terms contribute, taking $P$ fixed and sufficiently large (depending on $\varepsilon_0$),
\begin{align*}
&\ll k^2\,M^{1+\varepsilon}\sum_{\substack{Y\ll L\ll X\\\text{dyadic}}}L^{1+\varepsilon}\,M^{1-P\varepsilon_0/2}\ll k^2\, M^{2+\varepsilon-P\varepsilon_0/2}\,X^{1+\varepsilon} \ll1.
\end{align*}
Finally, the number of terms in $\sum_a\sum_b\sum_c\sum_d$ in which $\alpha\ll k\,M^{\varepsilon_0/2-1/2}$ is by Corollary \ref{spacing-corollary}
\[\ll L^{7/2+\varepsilon}\,k\,M^{\varepsilon_0/2-1/2}+L^{2+\varepsilon},\]
and so the contribution from these terms, estimating by absolute values, is
\begin{align*}
&\ll k^2\,M^{1+\varepsilon}\sum_{\substack{Y\ll L\ll X\\\text{dyadic}}}\left(L^{7/2+\varepsilon}\,k\,M^{\varepsilon_0/2-1/2}+L^{2+\varepsilon}\right)L^{\varepsilon-3}\,M\\
&\ll k^2\,M^{2+\varepsilon}\sum_{\substack{Y\ll L\ll X\\\text{dyadic}}}\left(L^{1/2}\,k\,M^{\varepsilon_0/2-1/2}+L^{-1}\right).
\end{align*}
The contribution from the second term $L^{-1}$ is
\[
\ll\begin{cases} k^2\,M^{2+\varepsilon} & \textrm{if }\,k\ll\Delta\, M^{-1/2} \\ \Delta^2\,M^{1+\varepsilon} & \textrm{otherwise.}\end{cases}
\]
Let us now move to considering the first term. In the case $k\ll \Delta\, M^{-1/2},$ we have $X=M^{1/2}$, and thus obtain
\[
\ll k^3\,M^{3/2+\varepsilon_0/2+\varepsilon}\,X^{1/2}\ll k^3\,M^{3/2+\varepsilon_0}\,M^{1/4}\ll k^3\,M^{7/4+\varepsilon_0}\ll k^2\,M^{2+\varepsilon_0},
\]
since $k\ll M^{1/4}$.

In the case $k\gg \Delta\, M^{-1/2}$, we have $X\ll\Delta^4\,M^{-1}\,k^{-6}$, and hence, the contribution is
\[
\ll k^3\,M^{3/2+\varepsilon_0/2+\varepsilon}\,X^{1/2}\ll k^3\,M^{3/2+\varepsilon_0}(\Delta^4\,M^{-1}\,k^{-6})^{1/2}\ll \Delta^2\,M^{1+\varepsilon_0}.
\]

\section{Proof of Theorem \ref{large-values}}

We begin by observing that we may assume $M$ to be larger than a fixed large constant, because when $M\ll1$, we also have $k\asymp\Delta\asymp V\asymp1\asymp M$ and the desired estimate for $R$ reduces to $R\ll1$, which would hold as certainly $R\ll 1+ M/V\ll1$ in this case. Also, in the following all implicit constants are allowed to depend on $\delta$ and $\left\langle p,q\right\rangle$. We also make the simple observation that we may assume that $V\ll\sqrt M$ for the sums in question cannot obtain larger values by the Wilton--Jutila estimate.

Let $x\in\left[M,2M\right]$, and let $N\in\mathbb R_+$ with $1\ll N\ll M$. We will choose $N$ later.
The truncated Voronoi identity says that
\begin{align*}
&\sum_{x\leq n\leq x+\Delta}a(n)\,e\!\left(n\,\frac{h}k\right)\\
&=\frac{k^{1/2}}{\pi\,\sqrt2}\sum_{n\leq N}a(n)\,n^{-3/4}\,e\!\left(-n\,\frac{\overline h}k\right)\\
&\quad\quad\cdot
\left((x+\Delta)^{1/4}\cos\!\left(4\pi\,\frac{\sqrt{n(x+\Delta)}}k-\frac\pi4\right)
-x^{1/4}\cos\!\left(4\pi\,\frac{\sqrt{nx}}k-\frac\pi4\right)\right)\\
&\quad+O(k\,M^{1/2+\delta}\,N^{-1/2}).
\end{align*}
If $x$ happens to be an integer, then the term $a(x)\,e(xh/k)$ is certainly $\ll x^\delta$ by Deligne's estimate, and this is certainly $\ll k\,M^{1/2+\delta}\,N^{-1/2}$.
Replacing the factor $(x+\Delta)^{1/4}$ by $x^{1/4}$ causes the error
\[\ll k^{1/2}\sum_{n\leq N}\left|a(n)\right|n^{-3/4}\,\Delta\,M^{-3/4}
\ll k^{1/2}\,N^{1/4}\,\Delta\,M^{-3/4}.\]
Also, the difference of the cosines may be replaced by a sine integral:
\begin{multline*}
\cos\!\left(4\pi\,\frac{\sqrt{n(x+\Delta)}}k-\frac\pi4\right)
-\cos\!\left(4\pi\,\frac{\sqrt{nx}}k-\frac\pi4\right)\\
=-\int\limits_x^{x+\Delta}\frac{2\pi\,\sqrt n}{k\,\sqrt t}\,\sin\!\left(4\pi\,\frac{\sqrt{nt}}k-\frac\pi4\right)\mathrm dt.
\end{multline*}
Combining the facts above gives
\begin{align*}
&\sum_{x\leq n\leq x+\Delta}a(n)\,e\!\left(n\,\frac{h}k\right)\\
&=-\sqrt2\,k^{-1/2}\,x^{1/4}\sum_{n\leq N}a(n)\,n^{-1/4}\,e\!\left(-n\,\frac{\overline h}k\right)
\int\limits_x^{x+\Delta}\sin\!\left(4\pi\,\frac{\sqrt{nt}}k-\frac\pi4\right)\frac{\mathrm dt}{\sqrt t}\\
&\qquad+O(k^{1/2}\,N^{1/4}\,\Delta\,M^{-3/4})+O(k\,M^{1/2+\delta}\,N^{-1/2}).
\end{align*}

We will split the interval $\left[M,2M\right]$ into $\ll1+M/M_0$ closed subintervals of length at most $M_0\in\mathbb R_+$ which we allow to have only endpoints in common. We shall choose the precise value of $M_0$ later. Also, we shall focus on one of the subintervals, say $J=\left[M,2M\right]\cap\left[\alpha,\alpha+\Lambda\right]$, where $\alpha\in\left[M,2M\right]$ and $\Lambda\in\left]0,M_0\right]$, which we assume to contain exactly $R_0\in\mathbb Z_+$ of the original points $x_1$, \dots, $x_R$. Without loss of generality, we may assume these points to be $x_1$, \dots, $x_{R_0}$, ordered so that $x_1<x_2<\ldots<x_{R_0}$. Once we have estimated $R_0$ from above as $\ll\Upsilon$, where $\Upsilon$ does not depend on $J$ but only on $k$, $M$, $\Delta$, $\delta$ and $V$, we can estimate $R$ from above by
\[R\ll\Upsilon\left(1+\frac M{M_0}\right).\]
Of course, if the subinterval contains none of the original points, then it trivially contains $\ll\Upsilon$ points.

Let us consider the choice of $N$ in the truncated Voronoi identity. Provided that
\[N\ll M^3\,V^4\,\Delta^{-4}\,k^{-2}\quad\text{and}\quad
N\gg k^2\,M^{1+2\delta}\,V^{-2},\]
where the former implicit constant needs to be sufficiently small and the latter sufficiently large,
the two error terms can be absorbed to the left-hand side, which in turn is $\gg V$, and we get for each $r\in\left\{1,\ldots,R_0\right\}$ the estimate
\begin{align*}
V&\ll\sum_{x_r\leq n\leq x_r+\Delta}a(n)\,e\!\left(n\,\frac{h}k\right)\\
&\ll k^{-1/2}\,M^{-1/4}\int\limits_{x_r}^{x_r+\Delta}\,\left|\sum_{n\leq N}
a(n)\,n^{-1/4}\,e\!\left(-n\,\frac{\overline h}k\right)
\sin\!\left(4\pi\,\frac{\sqrt{nt}}k-\frac\pi4\right)\right|\mathrm dt.
\end{align*}

We shall actually choose $N$ to be as small as possible, namely $N=c\,k^2\,M^{1+2\delta}\,V^{-2}$ with a fixed constant $c\in\mathbb R_+$, though dependent on $\delta$, and sufficiently large so that we can indeed absorb the term $k\,M^{1/2+\delta}\,N^{-1/2}$ to the left-hand side. We will have $N\ll M^{3-\delta}\,V^4\,\Delta^{-4}\,k^{-2}$ since $V\gg k^{2/3}\,\Delta^{2/3}\,M^{-1/3+\delta}$, and so $N\ll M^3\,V^4\,\Delta^{-4}\,k^{-2}$ with a very small implicit constant, provided that $M$ is sufficiently large.
The requirement $N\ll M$ is satisfied thanks to the condition $V\gg k\,M^{2\delta}$.
Similarly, the requirement $N\gg1$ is satisfied thanks to the condition $V\ll k\,M^{1/2+\delta}$. Also, we point out that, assuming that $M$ is sufficiently large, we may assume that $N\geq1$ for if $N<1$, then $V\gg k\,M^{1/2+\delta}$, and we would have
\[k\,M^{1/2+\delta}\ll V\ll\sum_{x_1\leqslant n\leqslant x_1+\Delta}a(n)\,e\!\left(n\,\frac hk\right)\ll M^{1/2},\]
which is not possible for large~$M$.

Next we cover the interval $J$ with consecutive semiclosed intervals
\[I_1=\left[\alpha,\alpha+V\right[,\,
I_2=\left[\alpha+V,\alpha+2V\right[,\,\ldots,\,
I_\nu=\left[\alpha+\left(\nu-1\right)V,\alpha+\nu V\right[,\]
where the number of intervals $\nu\in\mathbb Z_+$ is chosen so that it satisfies simultaneously the conditions $\nu\geq2\,R_0$, $\nu>\left(M+\Delta\right)/V+1$ as well as $\nu\ll M/V$. Let us temporarily simplify notation by writing
\[\Sigma(t)=\sum_{n\leq N}a(n)\,n^{-1/4}\,e\!\left(-n\,\frac{\overline h}k\right)\sin\!\left(4\pi\,\frac{\sqrt{nt}}k-\frac\pi4\right).\]
Let us consider integers
\[1\leq a_1<a_2<a_3<\ldots<a_{R_0}\leq\nu\]
such that
\[x_1\in I_{a_1},\quad x_2\in I_{a_2},\quad\ldots,\quad x_{R_0}\in I_{a_{R_0}},\]
and let $L=1+\left\lceil\Delta/V\right\rceil$ so that
\[\left[x_r,x_r+\Delta\right]\subseteq I_{a_r}\cup I_{a_r+1}\cup\ldots\cup I_{a_r+L},\]
for each $r\in\left\{1,2,\ldots,R_0\right\}$.
Furthermore, let $t_1\in I_1$, $t_2\in I_2$, \dots, $t_\nu\in I_\nu$ be points such that
\[\left|\Sigma(t_\ell)\right|=\max_{t\in\overline{I_\ell}}\left|\Sigma(t)\right|\]
for each $\ell\in\left\{1,2,\ldots,\nu\right\}$. Now we may continue by estimating
\begin{align*}
V&\ll k^{-1/2}\,M^{-1/4}\sum_{\ell=a_r}^{a_r+L}\int\limits_{I_\ell}\left|\Sigma(t)\right|\mathrm dt
\ll k^{-1/2}\,M^{-1/4}\sum_{\ell=a_r}^{a_r+L}V\left|\Sigma(t_\ell)\right|.
\end{align*}

Next, let us pick odd indices
\[1\leq v_1<v_2<\ldots<v_{R_0}\leq\nu\]
and even indices
\[2\leq w_1<w_2<\ldots<w_{R_0}\leq\nu\]
so that the absolute values $\left|\Sigma(t_{v_\ell})\right|$ for $\ell\in\left\{1,2,\ldots,R_0\right\}$ are the $R_0$ largest, counting multiplicities, among
\[\left|\Sigma(t_1)\right|,\quad\left|\Sigma(t_3)\right|,\quad\left|\Sigma(t_5)\right|,\quad\ldots,\]
and similarly, so that the absolute values $\left|\Sigma(t_{w_\ell})\right|$ are the $R_0$ largest, counting multiplicities, among
\[\left|\Sigma(t_2)\right|,\quad\left|\Sigma(t_4)\right|,\quad\left|\Sigma(t_6)\right|,\quad\ldots\]
Then we may continue our estimations by
\begin{align*}
R_0&\ll k^{-1/2}\,M^{-1/4}\sum_{r=1}^{R_0}\sum_{\ell=a_r}^{a_r+L}\left|\Sigma(t_\ell)\right|\\
&\ll k^{-1/2}\,M^{\delta-1/4}\,\Delta\,V^{-1}\sum_{\ell=1}^{R_0}\left|\Sigma(t_{v_\ell})\right|
+k^{-1/2}\,M^{\delta-1/4}\,\Delta\,V^{-1}\sum_{\ell=1}^{R_0}\left|\Sigma(t_{w_\ell})\right|,
\end{align*}
where the last estimate follows straightforwardly from the fact that the sums over $\ell$ intersect by at most $L+1$ terms and $L+1\ll M^\delta\,\Delta\,V^{-1}$ thanks to the condition $V\ll\Delta\,M^\delta$.
Without loss of generality and to simplify notation, we may assume that the term involving $v_\ell$ is larger, and we therefore can strike out here the terms involving $w_\ell$, at the price of an extra constant factor~$2$. Now, by the Cauchy--Schwarz inequality,
\[R_0\ll k^{-1/2}\,M^{\delta-1/4}\,\Delta\,V^{-1}\sqrt{R_0}\,\sqrt{\sum_{\ell=1}^{R_0}\left|\Sigma(t_{v_\ell})\right|^2},\]
so that
\[R_0\ll k^{-1}\,M^{2\delta-1/2}\,\Delta^2\,V^{-2}\sum_{\ell=1}^{R_0}\left|\Sigma(t_{v_\ell})\right|^2.\]
We split the sum $\Sigma(\cdot)$ dyadically, and write $\sin$ in terms of $e(\pm\ldots)$.  Then we continue by applying  Bombieri's lemma, and estimating $\log M\ll M^\delta$,
\begin{align*}
&R_0\ll k^{-1}\,M^{2\delta-1/2}\,\frac{\Delta^2}{V^{2}}\sum_{r=1}^{R_0}\left|\sum_{n\leq N}\frac{a(n)}{n^{1/4}}\,e\!\left(-n\,\frac{\overline h}k\right)\sin\!\left(4\pi\,\frac{\sqrt{nt_{v_r}}}k-\frac\pi4\right)\right|^2\\
&\ll k^{-1}\,M^{3\delta-1/2}\,\frac{\Delta^2}{V^{2}}\sum_\pm\sum_{r\leq R_0}\sum_{\substack{U\leq N/2\\\text{dyadic}}}\left|\sum_{U<n\leq2U}\frac{a(n)}{n^{1/4}}\,e\!\left(-n\,\frac{\overline h}k\right)e\!\left(\pm\frac{2\sqrt{nt_{v_r}}}k\right)\right|^2\\
&\ll k^{-1}\,M^{4\delta-1/2}\,\frac{\Delta^2}{V^{2}}\sum_\pm\max_{U\leq N/2}\sum_{r\leq R_0}\left|\sum_{U<n\leq2U}\frac{a(n)}{n^{1/4}}\,e\!\left(-n\,\frac{\overline h}k\right)e\!\left(\pm\frac{2\sqrt{nt_{v_r}}}k\right)\right|^2\\
&\ll k^{-1}\,M^{4\delta-1/2}\,\frac{\Delta^2}{V^{2}}\max_{U\leq N/2}U^{1/2}\max_{r\leq R_0}\sum_{s=1}^{R_0}\left|\sum_{U<n\leq2U}e\!\left(\frac{2\,\sqrt n\left(\sqrt{t_{v_r}}-\sqrt{t_{v_s}}\right)}k\right)\right|.
\end{align*}
The terms with $s=r$ are easily seen to contribute
\[\ll k^{-1}\,M^{4\delta-1/2}\,\Delta^2\,V^{-2}\,N^{3/2}
\ll k^2\,M^{1+7\delta}\,\Delta^2\,V^{-5}.\]
To estimate the remaining terms, those with $s\neq r$, we first observe that
\[\bigl|\sqrt{t_{v_r}}-\sqrt{t_{v_s}}\bigr|\asymp\int\limits_{t_{v_s}}^{t_{v_r}}\frac{\mathrm dt}{\sqrt t}\asymp\frac{\left|t_{v_r}-t_{v_s}\right|}{M^{1/2}}\ll\frac{M_0}{M^{1/2}},\]
and so we may use the theory of exponent pairs to estimate
\begin{align*}
&\sum_{U<n\leq2U}e\!\left(\frac{2\,\sqrt n\left(\sqrt{t_{v_r}}-\sqrt{t_{v_s}}\right)}k\right)\\
&\qquad\ll k^{-p}\,\bigl|\sqrt{t_{v_r}}-\sqrt{t_{v_s}}\bigr|^p\,U^{q-p/2}+\frac{k\,U^{1/2}}{\left|\sqrt{t_{v_r}}-\sqrt{t_{v_s}}\right|}\\
&\qquad\ll k^{-p}\,M_0^p\,M^{-p/2}\,U^{q-p/2}+\frac{k\,U^{1/2}\,M^{1/2}}{\left|t_{v_r}-t_{v_s}\right|}.
\end{align*}
Thus, the remaining terms contribute, estimating again $\log M\ll M^\delta$ and remembering that $q\geq1/2\geq p$ so that $1/2+q-p/2>0$,
\begin{align*}
&\ll k^{-1}\,M^{4\delta-1/2}\,\Delta^2\,V^{-2}\\
&\qquad\cdot\max_{U\leq N/2}U^{1/2}\left(R_0\,k^{-p}\,M_0^p\,M^{-p/2}\,U^{q-p/2}+\max_{r\leq R_0}\sum_{s\neq r}\frac{k\,U^{1/2}\,M^{1/2}}{\left|t_{v_r}-t_{v_s}\right|}\right)\\
&\ll k^{-1}\,M^{4\delta-1/2}\,\Delta^2\,V^{-2}\,N^{q+1/2-p/2}\,R_0\,k^{-p}\,M_0^p\,M^{-p/2}\\
&\qquad+k^{-1}\,M^{5\delta-1/2}\,\Delta^2\,V^{-2}\,N\,k\,M^{1/2}\,V^{-1}\\
&\ll R_0\cdot k^{2q-2p}\,\Delta^2\,M_0^p\,M^{q-p+\delta\left(5+2q-p\right)}\,V^{p-2q-3}
+k^2\,M^{1+7\delta}\,\Delta^2\,V^{-5}.
\end{align*}
We shall choose $M_0$ to be as large as possible so that the first term on the right-hand side will be $\ll R_0$ with a small implicit constant and can therefore be absorbed to the left-hand side. That is, we shall choose
\[M_0\asymp k^{2-2q/p}\,\Delta^{-2/p}\,M^{1-q/p+\delta\left(1-2q/p-5/p\right)}\,V^{2q/p-1+3/p}.\]
Thus, we have estimated $R_0$ as
\[\ll k^2\,M^{1+7\delta}\,\Delta^2\,V^{-5}.\]
The total estimate for $R$ is therefore
\begin{align*}
R&\ll k^2\,M^{1+7\delta}\,\Delta^2\,V^{-5}\left(1+\frac M{M_0}\right)\\
&\ll k^2\,M^{1+7\delta}\,\Delta^2\,V^{-5}\\
&\qquad+k^2\,M^{1+7\delta}\,\Delta^2\,V^{-5}\,M\,k^{2q/p-2}\,\Delta^{2/p}\,M^{-1+q/p+\delta\left(5/p+2q/p-1\right)}\,V^{-2q/p+1-3/p}\\
&\ll k^2\,M^{1+7\delta}\,\Delta^2\,V^{-5}
+k^{2q/p}\,\Delta^{2+2/p}\,M^{1+q/p+\delta\left(6+5/p+2q/p\right)}\,V^{-2q/p-4-3/p}.
\end{align*}

\section{Proof of Theorem \ref{moments-from-large-values}}

To estimate the integral
\[\int\limits_M^{2M}\left|\sum_{x\leq n\leq x+\Delta}a(n)\,e\!\left(n\,\frac{h}k\right)\right|^A\mathrm dx,\]
we estimate it separately in the regions where the integrand is $\leq \left(M^{2\delta}\,V_0\right)^A$ and $\geq \left(M^{2\delta}\,V_0\right)^A$, where $\delta\in\mathbb R_+$ is small and fixed. The former values contribute $\ll M^{1+2\delta A}\,V_0^A$. To estimate the contribution from the latter values, we split the remaining value range dyadically into intervals of the shape $\left[V,2V\right]$ with $V\in\left[V_0\,M^{2\delta},\infty\right[$. If necessary, we extend the last interval, losing at most a constant factor in the estimations. The number of subintervals is $\ll\log M\ll M^\varepsilon$. For each value interval, we choose a maximal number of points $x_1$, \dots, $x_{R(V)}$ from the interval $\left[M,2M\right]$ so that
\[\left|\sum_{x_r\leq n\leq x_r+\Delta}a(n)\,e\!\left(n\,\frac{h}k\right)\right|\in\left[V,2V\right]\]
for each $r\in\left\{1,\ldots,R(V)\right\}$ and that $\left|x_r-x_s\right|\geq V$ for all $r,s\in\left\{1,\ldots,R(V)\right\}$ with $r\neq s$. We recall that we certainly have $R(V)=0$ if $V\gg\sqrt M$ or $V\gg \Delta\,M^{\delta}$ or $V\gg k^\alpha\,\Delta^\beta\,M^{\gamma+2\delta}$. Now the contribution from the large values of the integrand is bounded by
\[\ll_A\sum_VV\cdot R(V)\,V^A,\]
where the summation over $V$ is dyadic. Using Theorem \ref{large-values}, this is
\begin{multline*}\ll_\delta\sum_V\bigl(k^2\,M^{1+7\delta}\,\Delta^2\,V^{-5}
\bigr.\\\bigl.+k^{2q/p}\,\Delta^{2+2/p}\,M^{1+q/p+\delta\left(6+5/p+2q/p\right)}\,V^{-2q/p-4-3/p}\bigr)V^{A+1}.\end{multline*}
In each term $V$ is estimated from below by $V_0$ or from above by $k^\alpha\,\Delta^\beta\,M^{\gamma+2\delta}$, depending on whether the final exponent of $V$ is negative or positive.
Upon letting $\delta$ have smaller and smaller values, the first term in the parentheses gives rise to $\Phi$ and the second to $\Psi$.

\section{Proof of Theorem \ref{example-bound1}}

\begin{proof}[Proof of Theorem \ref{example-bound1}]
We choose $p=q={1}/{2}$. By Theorem 5.5 in \cite{Ernvall-Hytonen--Karppinen2008} $$\sum_{x\leq n\leq x+\Delta}a(n)\,e\!\left(n\,\frac{h}{k}\right)\ll\Delta^{1/6}\,M^{1/3+\varepsilon}
\ll M^{5/(12\cdot6)}\,M^{1/3+\varepsilon}\ll M^{29/72+\varepsilon}.$$ Notice that when $k\gg M^{1/9}$, this bound is superior to $k^{1/4}\,M^{3/8+\varepsilon}$ from \cite{Vesalainen2014}. Hence using Theorem \ref{moments-from-large-values}, we obtain, for any $V_0\in\left[1,\infty\right[$ with $k\ll V_0\ll M^{29/72}$ and $V_0\gg k^{2/3}\,\Delta^{2/3}\,M^{-1/3}$, that
\begin{multline*}
\int\limits_M^{2M}\left|\sum_{x\leq n\leq x+\Delta}a(n)\,e\!\left(n\,\frac{h}{k}\right)\right|^Adx\\ \ll M^{1+\varepsilon}\,V_0^A+k^2\,M^{11/6+29(A-4)/72+\varepsilon}+k^2\,M^{9/2+29(A-11)/72+\varepsilon}.
\end{multline*}
The term $V_0$ does not appear anywhere else except in the main term, so we can choose it to be as small as possible, namely $k$. For this choice, we also have $V_0\gg k^{2/3}\,M^{-1/18}=k^{2/3}\,\Delta^{2/3}\,M^{-1/3}$. The contribution of the main term is $k^A\,M^{1+\varepsilon}$. The three terms satisfy
\[
k^2\,M^{11/6+29(A-4)/72+\varepsilon}\gg k^2\,M^{9/2+29(A-11)/72+\varepsilon},\]
and
\[k^2\,M^{11/6+29(A-4)/72+\varepsilon}\gg k^A\,M^{1+\varepsilon},\] and we get the claimed bound.
\end{proof}

\section{Proof of Theorem \ref{sovellus2}}

\begin{proof}[Proof of Theorem \ref{sovellus2}] We will first apply Theorem \ref{esa_moments} with the exponent pair $p={4}/{18}$ and $q={11}/{18}$ and the parameters $\alpha=\gamma=\delta=0$, $k=1$ and $\beta={1}/{3}$. Now the main term becomes $M^{A/4+1}$. The term $\Phi$ becomes $M^{(A+1)/3+\varepsilon}$ and the term $\Psi$ becomes $M^{A/4+1+\varepsilon}$. Since $\left({A+1}\right)/{3}<{A}/{4}+1$ exactly when $A<8$, we have now derived
\begin{equation*}
\int\limits_M^{2M}\left|\sum_{x\leq n\leq x+\Delta}a(n)\right|^Adx\ll \\ \begin{cases}M^{A/4+1+\varepsilon} & \textrm{when }A\leq 8,\\ M^{(A+1)/3+\varepsilon} & \textrm{when }A\geq 8.\end{cases}
\end{equation*}

We will now apply Theorem \ref{moments-from-large-values} with exponent pair $p=q={1}/{2}$. By the trivial estimate and by the estimate for a long sum, we know that
\[
\sum_{x\leq n\leq x+\Delta}a(n)\ll \min \left(\Delta\, M^{\varepsilon},M^{1/3+\varepsilon}\right).
\]
Using these bounds we obtain
\begin{multline*}
\int\limits_M^{2M}\left|\sum_{x\leq n\leq x+\Delta}a(n)\right|^Adx\\
\ll
\begin{cases}M^{1+\varepsilon}\,V_0^{A}+\Delta^{A-2}\,M^{1+\varepsilon}+\Delta^{6}\,M^{2+\varepsilon}\,V_0^{A-11}&\text{when $\Delta\ll M^{1/3}$},\\
M^{1+\varepsilon}\,V_0^{A}+\Delta^2\,M^{(A-1)/3+\varepsilon}+\Delta^{6}\,M^{2+\varepsilon}\,V_0^{A-11}&\text{when $\Delta\gg M^{1/3}$}.
\end{cases}
\end{multline*}

Let us now choose $V_0$ so that the first and the last term are the same (up to an epsilon):
\[
M\,V_0^{A}=\Delta^{6}\,M^{2}\,V_0^{A-11},
\]
which is equivalent with $V_0=\Delta^{6/11}\,M^{1/11}$. Clearly, this choice satisfies $V_0\gg1=k$ and is easily seen to satisfy $V_0\gg \Delta^{2/3}\,M^{-1/3}$. It also satisfies $V_0\ll \Delta\,M^\varepsilon$ and $V_0\ll M^{1/3+\varepsilon}$ since $M^{1/5}\ll\Delta\ll M^{4/9}$. The estimate becomes now
\begin{multline*}
\int\limits_M^{2M}\left|\sum_{x\leq n\leq x+\Delta}a(n)\right|^Adx
\\
\ll \begin{cases}
M^{1+A/11+\varepsilon}\,\Delta^{6A/11}+\Delta^{A-2}\,M^{1+\varepsilon}&\text{when $
\Delta\ll M^{1/3}$,}\\
\Delta^2 \,M^{(A-1)/3+\varepsilon}+M^{1+A/11+\varepsilon}\,\Delta^{6A/11}&\text{when $\Delta\gg M^{1/3}$}.
\end{cases}\end{multline*}

When $\Delta\ll M^{1/3}$, we have $\Delta^{5A/11-2}\ll M^{A/11}$ and hence $M^{A/11}\,\Delta^{6A/11}\gg \Delta^{A-2}$. Thus, $M^{1+A/11+\varepsilon}\,\Delta^{6A/11}\gg \Delta^{A-2}\,M^{1+\varepsilon}$.

When $\Delta \gg M^{1/3}$, we have
$\Delta\gg M^{1/3}\gg M^{(4A-22)/(9A-33)}$, so that 
$\Delta^2\, M^{(A-1)/3+\varepsilon}\ll \Delta^{6A/11}M^{1+A/11+\varepsilon}$. We have now derived
\begin{equation*}
\int\limits_M^{2M}\left|\sum_{x\leq n\leq x+\Delta}a(n)\right|^Adx\ll 
M^{1+A/11+\varepsilon}\,\Delta^{6A/11}.
\end{equation*}
Finally, the proof is completed by comparing the above bounds separately in the cases $A\geq8$ and $A\leq8$.
\end{proof}

\section*{Acknowledgements}

The first author was funded by the Academy of Finland project 138337, by the Finnish Cultural Foundation, and by the Ruth och Nils-Erik Stenb\"acks stiftelse. The second author was funded by the Academy of Finland through the Finnish Centre of Excellence in Inverse Problems Research and the projects 276031, 282938, 283262 and 303820, by the Magnus Ehrnrooth Foundation, by the Finnish Cultural Foundation, by the Foundation of Vilho, Yrj\"o and Kalle V\"ais\"al\"a, and by the Basque Government through the BERC 2014--2017 program and by Spanish Ministry of Economy and Competitiveness MINECO: BCAM Severo Ochoa excellence accreditation SEV-2013-0323.

\end{document}